\documentclass[11pt]{amsart} 
\usepackage{enumerate} 
\usepackage{amsmath,amsfonts, amscd,  eucal, latexsym, amssymb,   rotating}
 \usepackage[utf8]{inputenc} 
 \usepackage[T1]{fontenc}
 \usepackage{mathrsfs}

\DeclareFontFamily{OT1}{pzc}{}
\DeclareFontShape{OT1}{pzc}{m}{it}{<-> s * [1.200] pzcmi7t}{}
\DeclareMathAlphabet{\mathpzc}{OT1}{pzc}{m}{it}

\usepackage{tikz-cd}

 \usepackage[colorlinks=true, pdfstartview=FitV,
linkcolor=cyan, citecolor=red, urlcolor=blue]{hyperref}

\newif\ifPDF
\ifx\pdfoutput\undefined 
        \ifx\pdfversion\undefined
                \PDFfalse
        \fi
\else 
        \ifnum\pdfoutput > 0
                \PDFtrue
        \else
                \PDFfalse
        \fi
\fi

\makeatletter
\newcommand{\xRightarrow}[2][]{\ext@arrow 0359\Rightarrowfill@{#1}{#2}}
\makeatother

\theoremstyle{plain}
				
\newtheorem{theorem}{Theorem}[section]				
\newtheorem{proposition}[theorem]{Proposition}		
\newtheorem{corollary}[theorem]{Corollary}
\newtheorem{lemma}[theorem]{Lemma}

\theoremstyle{definition}

\newtheorem{definition}[theorem]{Definition}

\newcommand{\Mbold}{{\bf M}}

\newcommand{\Vbold}{{\bf V}}

\newcommand{\ebold}{{\bf e}}

\newcommand{\vbold}{{\bf v}}

\newcommand{\CBbb}{\mathbb C}

\newcommand{\RBbb}{\mathbb R}

\newcommand{\Acal}{\mathcal A}

\newcommand{\Gcal}{\mathcal G}

\newcommand{\Jcal}{\mathcal J}

\newcommand{\Ocal}{\mathcal O}

\newcommand{\Vcal}{\mathcal V}

\newcommand{\SL}{\mathsf{SL}}

\newcommand{\slfrak}{\mathfrak{sl}}

\DeclareMathOperator{\ad}{ad}
\DeclareMathOperator{\Ad}{Ad}
\DeclareMathOperator{\tr}{tr}

\newcommand{\dbar}{\bar\partial}

\newcommand{\lra}{\longrightarrow}

\newcommand{\Op}{{\bf Op}}

\newcommand{\isorightarrow}{\xrightarrow{
   \,\smash{\raisebox{-0.5ex}{\ensuremath{\sim}}}\,}}

\makeatletter
\newcommand*\bigcdot{\mathpalette\bigcdot@{.5}}
\newcommand*\bigcdot@[2]{\mathbin{\vcenter{\hbox{\scalebox{#2}{$\m@th#1\bullet$}}}}}
\makeatother

\makeatletter
\newcommand{\thickcolon}{\mathpalette\thick@colon\relax}
\newcommand{\thick@colon}[2]{%
  \mspace{1mu}%
  \vbox{%
    \hbox{$\m@th#1\bigcdot$}
    \nointerlineskip
    \kern.15ex
    \hbox{$\m@th#1\bigcdot$}
    \kern-.55ex
  }%
  \mspace{1mu}%
}
\makeatother

\makeatletter
\newcommand{\thickcirc}{\mathpalette\thick@circ\relax}
\newcommand{\thick@circ}[2]{%
  \mspace{1mu}%
  \vbox{%
    \hbox{$\m@th#1\circ$}
    \nointerlineskip
    \kern.15ex
    \hbox{$\m@th#1\circ$}
    \kern-.55ex
  }%
  \mspace{1mu}%
}
\makeatother

\makeatletter
\newcommand*\bigldot{\mathpalette\bigldot@{.5}}
\newcommand*\bigldot@[2]{\mathbin{\hbox{\scalebox{#2}{$\m@th#1\bullet$}}}}
\makeatother

\newcommand*{\dt}[1]{\overset{\bigldot}{#1}}

\begin{document}

\title[Discreteness of real $\SL_n$-opers]{Discreteness of real $\SL_n$-opers}
	
\author[Dimakis]{Panagiotis Dimakis}

\author[Wentworth]{Richard A. Wentworth}

\address{Department of Mathematics,
   University of Maryland,
   College Park, MD 20742, USA}
\email{pdimakis12345@gmail.com}
\email{raw@umd.edu}
\thanks{R.W. supported in part by NSF grant DMS-2506596.}


\begin{abstract}
 We prove that the set of real $\SL_n$-opers forms a discrete subset of the moduli space
 of irreducible flat connections.
\end{abstract}

\maketitle

\allowdisplaybreaks

\thispagestyle{empty}

\section{Statement of the result}
Fix a compact Riemann surface $X$ of genus $g\geq 2$. Let $\Mbold_n$
denote the deRham moduli space of flat $\SL_n(\CBbb)$ connections on $X$
(see \cite{Simpson:94b}). 
This is a complex analytic variety of dimension $2(n^2-1)(g-1)$, and it is smooth at
the irreducible connections $\Mbold_n^{irr.}$.
There is an antiholomorphic involution on $\Mbold_n$ whose fixed point set
$\Mbold_n^\RBbb$ consists of connections with real holonomy. 

Let $\Op_n\subset \Mbold_n$ denote the set of $\SL_n$-opers \cite{BD:05}. Then $\Op_n$ 
is a half-dimensional submanifold of $\Mbold_n^{irr.}$. It  has $n^{2g}$
components, and each is naturally an affine space over the sum $\bigoplus_{k=2}^n
H^0(X,K_X^{\otimes k})$ of holomorphic $k$-differentials on $X$ (see
\cite{Wentworth:16} for the proofs of some of these facts in the spirit of
this paper).   
It is a theorem of Faltings \cite{Faltings:83} that for $n=2$, the intersection $\Op_2\cap
\Mbold_2^\RBbb$ is discrete (see also \cite{Goldman:87}).
The rank $n$ version of this statement was conjectured in \cite[Conj.\
1.11(1)]{EFK:23}. A partial result was obtained in \cite{Wakabayashi} for
opers arising from the principal $\slfrak_2$-embedding.  We shall
refer to points
in the intersection $\Op_n^\RBbb:=\Op_n\cap \Mbold_n^\RBbb$
as \emph{real opers}. The goal of this note
is to prove the following\footnote{During the writing of this paper
Prof.\ Etingof announced a proof of Corollary \ref{cor:main} in a talk in
Madrid and kindly shared the preprint \cite{EFK:26} with us.}. 

\begin{theorem} \label{thm:main}
    Let $[\nabla]$ be a real oper. Then the intersection $\Op_n\cap \Mbold_n^\RBbb$ is transverse at $[\nabla]$. 
\end{theorem}

\begin{corollary} \label{cor:main}
    For any $n\geq 2$, 
    $\Op_n^\RBbb\subset\Mbold_n^{irr.}$ is a discrete set.
\end{corollary}

\section{Preliminaries}

\subsection{Flat connections} \label{sec:flat}

Let $E=X\times\CBbb^n$ be a complex rank $n$ vector bundle with a fixed
trivialization. 
The bundle $\ad(E)$ of smooth  traceless endomorphisms inherits a trivialization as
well.
With respect to this trivialization, any connection
$\nabla$ may be expressed $\nabla=d+B$, where $B\in \Omega^1(X,\ad(E))$.
In a local conformal coordinate $z$ on $X$, $B=B^{1,0}(z)dz+B^{0,1}(z)d\bar z$.  Thus we have a well-defined
real structure given by:
$$
\overline\nabla:= d+\overline{B}\quad ,\quad 
\overline B=\overline{B^{0,1}}dz+\overline{B^{1,0}}d\bar z\ .
$$

Let $\Acal(E)$ denote the infinite dimensional space of connections on $E$.
The trivialization gives the identification $\Acal(E)=\Omega^1(X,\ad(E))$.
Let $\Gcal(E)$ denote the space of smooth gauge transformations. This is an
infinite dimensional Frechet  Lie group  identified under the
trivialization with the space of smooth maps $X\to \SL_n(\CBbb)$.  Denote
the action of $\Gcal(E)$ on $\Acal(E)$ by: $(g,\nabla)\mapsto g(\nabla)$.
Explicitly, 
$$
g(\nabla)=g\circ d_\nabla\circ g^{-1}\quad ,\quad B\longmapsto
gBg^{-1}-dgg^{-1}\ .
$$
Let $\Acal^{cr}(E)\subset\Acal(E)$ denote the subspace of \emph{completely
reducible} connections. This is preserved under the action of $\Gcal(E)$
and has closed orbits.
Then $\Mbold_n:=\Acal^{cr}(E)/\Gcal(E)$ can be given the structure of a complex
variety of dimension $2(n^2-1)(g-1)$. 
The open subset  $\Mbold_n^{irr.}\subset\Mbold_n$ of equivalence classes of
irreducible connections lies in the smooth locus. 
For $[\nabla]\in \Mbold_n^{irr.}$, the tangent space $T_{[\nabla]}\Mbold_n$
is identified with $H^1_{d_\nabla}(X,\ad(E))$.

\begin{definition}
    The \emph{real locus} is defined by
$
    \Mbold_n^{\RBbb}=\left\{ [\nabla]\in \Mbold_n \mid
    [\overline\nabla]=[\nabla]\right\}
$.
\end{definition}
\noindent
Explicitly,  $[\nabla]\in \Mbold_n^{\RBbb}$
if for one (and hence any) representative connection $\nabla\in [\nabla]$
there is $g\in \Gcal(E)$ such that $\overline\nabla=g(\nabla)$. 
The map $[\nabla]\mapsto [\overline\nabla]$ defines an antiholomorphic
involution of $\Mbold_n$ with $\Mbold_n^\RBbb$ as its fixed point set.

\subsection{Opers} \label{sec:opers}

Fix $n=2q-1$. Let $D$ be an $n$-th order differential operator realized in
$K_X^{1-q}$, where we fix a square root of $K_X$ if $n$ is even\footnote{As
mentioned above, not every component of $\Op_n$ comes from differential
operators in $K_X^{1-q}$; it will be clear, however, that the arguments in this paper
apply to the other components as well}.  The corresponding local system $\Vbold$ 
is given by the kernel of $D$, so that we
have an exact sequence of $\underline\CBbb$-sheaves:
$$
0\lra \Vbold\xrightarrow[\hspace{.75cm}]{\varphi}K_X^{1-q}
\xrightarrow[\hspace{.75cm}]{D} K_X^q\lra 0\ ,
$$
where $\underline\CBbb$ is the sheaf of locally constant functions.
Let $\Vcal=\Vbold\otimes_{\underline\CBbb}\Ocal_X$. Then there is an
isomorphism 
$$
\Phi : \Vcal\isorightarrow \Jcal^n(K_X^{1-q}) : f\otimes \vbold\longmapsto f
j^n(\varphi(\vbold)) \ ,
$$
where $\Jcal^n(K_X^{1-q})$ is the $n$-jet bundle of $K_X^{1-q}$, and
$j^n(s)$ is the $n$-jet of a local section $s$. 
In particular, the underlying bundle of an oper is fixed (in a given
component), and only the
differential operator $D$ varies.

We will need a more explicit form of this. Let
$\vbold=\{\vbold_1,\ldots,\vbold_n\}$ be a local frame of $\Vbold$. 
This gives a local holomorphic frame for $\Vcal$, which we denote also by
$\vbold$.
A
choice of local coordinate $z$ on $X$  gives local trivialization
$\ebold$ of
$\Jcal^n(K_X^{1-q})$, where for a section $s=s(z)dz^{1-q}$, $j^n(s(z))=(s(z),s'(z), \ldots,
s^{(n-1)}(z))$. Set $y_i=\varphi(\vbold_i)$.
Then the $y_i$ are solutions to the equation
$$
Dy=y^{(n)}+Q_2y^{(n-2)}+\cdots+Q_ny=0
$$
where $Q_k$ are local holomorphic functions.
Here, $12Q_2/n(n^2-1)$ defines a projective connection, and by the Wilczynski
transform, the $Q_k$'s for $k\geq 3$ are in 1-1 correspondence with
holomorphic $k$-differentials on $X$ \cite{Wil:62}. 

Let $\nabla$ denote the holomorphic connection on $\Vcal$ determined by
$D$.
By definition,   $\nabla = d$ in the frame
 $\vbold$.
In the frame $\ebold$ determined by the choice of
coordinate $z$, $\nabla$ 
is represented  by the companion matrix $\nabla=d+A(z)dz$, where:
\begin{equation} \label{eqn:A}
    A(z)=\left(\begin{matrix}0 &-1 &0 &\ldots &0 \\
    0 & 0 & -1 & \ldots  &0 \\
    \vdots &\ddots &\ddots &\ddots& \vdots \\
    0 & \ldots & \ldots & 0 &-1 \\
Q_n & Q_{n-1} & \ldots & Q_2 &0 \end{matrix}
\right)
\end{equation}
The gauge transformation relating 
these two frames is the Wronskian matrix:
\begin{equation} \label{eqn:wronskian}
W(z)=\left(\begin{matrix} y_1(z)&\cdots& y_n(z)\\
    y_1'(z) &\cdots & y_n'(z)\\
    \vdots & \cdots &\vdots \\
y_1^{(n-1)}(z) &\cdots &y_n^{(n-1)}(z)  \end{matrix}\right)
\end{equation}

Lift this to the universal cover $\widetilde X$, where $\Gamma=\pi_1(M)$ acts by
deck transformations. We may then assume the frames are global.
By definition of the
holonomy $\rho$ of $\nabla$:
\begin{equation} \label{eqn:holonomy}
    \vbold(\gamma z)=\rho(\gamma)\vbold(z)\ .
\end{equation}
The representation $\rho:\Gamma\to \SL_n(\CBbb)$ is well-defined up to conjugation, depending upon
the initial choice of frame at a base point.


Let $\Op_n\subset \Mbold_n$ denote the set of gauge equivalence classes of
flat connections coming from opers. It is an important fact  that
every oper is irreducible, and hence $\Op_n\subset \Mbold_n^{irr.}$. 

\subsection{The reality condition}

\begin{proposition}  \label{prop:real}
    Let $\nabla$ be irreducible and $\overline\nabla=g(\nabla)$.
    Then $g\overline g=\pm I$.
\end{proposition}

\begin{proof}
    In the trivial frame (see \S \ref{sec:flat}), write $\nabla=d+B$. Then we have
    $$
    \overline B=g B g^{-1} - dg g^{-1} \ .
    $$
    Breaking into type, we have
$$
    \overline {B^{0,1}}=g B^{1,0}g^{-1}-\partial g g^{-1}\quad ,\quad
    \overline{B^{1,0}}=g B^{0,1}g^{-1}-\dbar g g^{-1} \ ,
$$
and hence,
    \begin{align*}
        B^{0,1}&=-\dbar {\bar g}{\bar g}^{-1}+\bar g\overline{B^{1,0}}{\bar
    g}^{-1} 
        =-\dbar {\bar g}{\bar g}^{-1}+\bar g
        \left(g B^{0,1}g^{-1}-\dbar g g^{-1}\right){\bar g}^{-1} \\
        &=-\dbar k k^{-1}+k B^{0,1}k^{-1}\ ,
    \end{align*}
    where $k=\bar g g$.
    A similiar calculation shows that $B^{1,0}=-\partial k k^{-1}+k
    B^{1,0}k^{-1}$. Hence, $d_\nabla(k)=0$. Since $\nabla$ is irreducible,
    $k$ must be constant and in the center of $\SL_n(\CBbb)$. But then $k$ must
   also be real, and the conclusion follows. 
\end{proof}

The following is clear from the definition of holonomy.
\begin{lemma} \label{lem:real-holonomy}
    Let $\nabla$ satisfy $\overline\nabla=g(\nabla)$, and let $\rho$ denote
    the holonomy of $\nabla$.
    Then there is $g_0\in \SL_n(\CBbb)$ such that for all
    $\gamma\in\Gamma$,
    $\overline{\rho(\gamma)}=g_0\rho(\gamma)g_0^{-1}$.
\end{lemma}

\section{Proof of the main theorem}

\subsection{A hermitian form}
Let $\nabla$ be a real connection: $\overline\nabla=g(\nabla)$.  
We define a sesquilinear form on $\Omega^1(X,\ad(E))$ by the pairing:
\begin{equation} \label{eqn:pairing}
    H_\nabla(b_1,b_2)= i\int_X \tr(b_1\wedge g^{-1}\overline{b_2}g)
\end{equation}

\begin{lemma} \label{lem:form}
$H_\nabla$  defines a hermitian form  on 
    $T_{[\nabla]}\Mbold_n$ for any $[\nabla]\in \Mbold_n^\RBbb$.
\end{lemma}

\begin{proof}
   We first verify that $H_\nabla$ is hermitian. Indeed,
    using Proposition \ref{prop:real}, we have
    \begin{align*}
        \overline{H_\nabla(b_1,b_2)}&=-i\int_X \tr(\overline{b_1}\wedge
        \overline g^{-1}b_2\overline g) 
        =-i\int_X \tr(\overline{b_1}\wedge  gb_2 g^{-1}) 
        =-i\int_X \tr(g^{-1}\overline{b_1}g\wedge b_2 ) \\
        &=i\int_X \tr(b_2\wedge g^{-1}\overline{b_1}g) 
        =H_\nabla(b_2,b_1)\ .
\end{align*}
    Now suppose $b_1,b_2\in \ker d_\nabla$.
    Suppose that we change $b_2\to b_2+d_\nabla\eta$,
    for some endomorphism $\eta$.  Then 
    again using $\overline\nabla=g(\nabla)$, we have
    $\overline{b_2}\lra \overline{b_2}+d_{g(\nabla)}\overline\eta$.
    Now $d_{g(\nabla)}=\Ad_g\circ d_\nabla \circ \Ad_{g^{-1}}$ and
    $d_\nabla b_1=0$,  and so the change in the integrand is
    $$
    \tr(b_1\wedge d_\nabla(\Ad_{g^{-1}}\overline\eta))
    =-d\tr(b_1 g^{-1}\overline\eta g) \ .
    $$
    A similar argument holds for a change $b_1\to b_1+d_\nabla\eta$.
    Hence, $H_{\nabla}$ is independent of the choice of representatives of
    $[b_1], [b_2]\in H^1_{d_\nabla}(X,\ad(E)))$.  
    Finally, if $\nabla_1=g_1(\nabla)$, then
    $\overline\nabla_1=(\overline{g_1}gg_1^{-1})(\nabla_1)$. 
    Under the isomorphism $H^1_{d_\nabla}(X,\ad(E))\isorightarrow
    H^1_{d_{\nabla_1}}(X,\ad(E))$ given by $[b]\mapsto [\Ad_{g_1}b]$,
    we can compute that the hermitian form is invariant. This completes the proof.
\end{proof}

For simplicity, we continue to write $H_\nabla(b_1,b_2)$ for the form on
$T_{[\nabla]}\Mbold_n$.

\begin{lemma} \label{lem:real}
    Suppose $\nabla$ is an irreducible real connection.
    Let $[b]\in H^1_{d_\nabla}(X,\ad(E))$ be tangent to the real locus. Then
    $H_\nabla(b,b)=0$. 
\end{lemma}

\begin{proof}
Suppose we have a one parameter family $\nabla_t$ of connections satisfying 
    $\overline \nabla_t =g_t(\nabla_t)$, and set $b=\dt\nabla$, $h=\dt g
    g^{-1}$ at $t=0$.
    Then a calculation shows  that $\overline
    b=gbg^{-1}-d_{\overline\nabla}h$, or
    $$
    g^{-1}\overline b g=b-\Ad_{g^{-1}}d_{g(\nabla)}h=b-d_\nabla(g^{-1}hg) \
    .
    $$
Pairing this with $b$ in $H_\nabla$ gives zero, since 
$$
    \tr(b\wedge g^{-1}\overline b g)=\tr(b\wedge (b-d_\nabla(g^{-1}hg)))=
    d\tr(b\wedge g^{-1}hg) \ , 
$$
which then vanishes upon integration.
\end{proof}

On the other hand, $H_\nabla$ is non-degenerate in the oper directions.

\begin{proposition} \label{prop:definite}
    Suppose $\nabla$ is a real oper. Then $H_\nabla$ is definite on
    $T_{[\nabla]}\Op_n$ (with the sign depending on the sign of $g\overline
    g$ in Proposition  \ref{prop:real}.
\end{proposition}

\begin{proof}
    Let $[b]\in T_{[\nabla]}\Op_n$. We choose a particular representative.
    In the frame $\ebold$, a variation of opers $\nabla=d+A$ is given by a family
    $Q_i+tq_i$, where $q_i$ are holomorphic $i$-differentials (locally). 
    Let $a=\dt{A}$. We may write this as: $a=\ebold_n\otimes q^t$, where 
$$
    q^t=\sum_{j=1}^{n-1} q_{n-j+1}\ebold_j^\ast
$$
in the dual frame $\ebold^\ast$ to $\ebold$.
    In the frame $\vbold$, 
    $b=W^{-1} a W= (W^{-1}\ebold_n)\otimes q^t W $,
    where $W$ is as in \eqref{eqn:wronskian}.
    By Proposition \ref{prop:real}, $g\overline g=\epsilon I$, where
    $\epsilon^2=1$. Then:
    \begin{align*}
        i\tr(b\wedge g^{-1}\overline b g) &=
        i\tr\left((W^{-1}\ebold_n)\otimes q^t W \wedge
        g^{-1}(\overline{W^{-1}\ebold_n})\otimes \overline{q^t W} g
        \right) \\
        &=
        i\epsilon\tr\left((gW^{-1}\ebold_n)\otimes q^t W \wedge
        (\overline{gW^{-1}\ebold_n})\otimes \overline{q^t W} 
        \right)  \ .
    \end{align*}
    But in local coordinates, this last line may be written
    \begin{equation} \label{eqn:square} 
    =\epsilon \left| q^t W(\overline{gW^{-1}\ebold_n})\right|^2\, idz\wedge
    d\bar z \ .
    \end{equation} 
    This proves the semidefiniteness.  Assume now that
    $H_{\nabla}(b,b)=0$.  From \eqref{eqn:square}, 
$$q^t W(\overline{gW^{-1}\ebold_n})=0\ .$$
We wish to show that this implies $q=0$.
Using the induced connection on the dual bundle, along with
    $\overline\nabla=g(\nabla)$, we have
    \begin{align}
        0=d\left(q^t W(\overline{gW^{-1}\ebold_n})\right)
        &=d_\nabla(q^t W)(\overline{gW^{-1}\ebold_n})
        + (q^t W)(\overline{d_{\overline\nabla}(gW^{-1}\ebold_n}))\notag \\
        0&=d_\nabla(q^t W)(\overline{gW^{-1}\ebold_n})
        + (q^t W)(\overline{gd_{\nabla}(W^{-1}\ebold_n})) \label{eqn:key}
    \end{align}
    Taking the $(0,1)$ part in \eqref{eqn:key}, the first term vanishes
    because $q$ and $W$ are holomorphic. Since in the frame $\vbold$ the
    connection is trivial, we find that 
    $ (q^t W)(\overline{g{U'}})=0$,
    where $U=W^{-1}\ebold_n$.  Repeating this for the $k$-th derivatives
    $U^{(k)}$, we have 
    $q^tW(\overline{gU^{(k)}})=0$, $k=0,\ldots, n-1$.
    Now the result follows from Lemma \ref{lem:independent} below.
\end{proof}

\begin{lemma} \label{lem:independent}
    The vectors  $\{U,U', \ldots,  U^{(n-1)}\}$ form a basis at
    every point.
\end{lemma}

\begin{proof}
    By the discussion in \S \ref{sec:opers}, 
    we have $dW W^{-1}=-A$, or $d(W^{-1})=W^{-1}A$, where $A$ is given by
    \eqref{eqn:A}. For $j=1,\ldots, n$, set $U_j=W^{-1}\ebold_j$, so that $U=U_n$.  
    Then a computation shows that for all $j$,
    $$
     U_j^\prime=-U_{j-1}+Q_{n-j+1}U_n\ .
    $$
    where for convenience we set $Q_1=0$, $U_0=0$.
    It follows inductively that 
$$
    U_j^{(k)}=(-1)^kU_{j-k}+\sum_{i\geq n-i+1}f_i U_i \ .
    $$
    for functions $f_i$ involving $Q_j$'s and derivatives.
Hence, in particular, 
    \begin{equation} \label{eqn:matrix} 
    \left(\begin{matrix} U \\  U' \\ \vdots \\
    U^{(n-1)}\end{matrix}\right)=
    \left(\begin{matrix} 1 & 0 & \ldots & 0 \\
        0 & -1 & \ldots & 0 \\
        \ast &\ast & \ddots & \vdots \\
    \ast & \ast & \ast & (-1)^{n-1} \end{matrix}
    \right)
    \left(\begin{matrix} U_n \\  U_{n-1} \\ \vdots \\
    U_1\end{matrix}\right)
    \end{equation} 
Since $W$ is invertible, the vectors $U_1,\ldots, U_n$ are pointwise
    independent. Since the matrix in \eqref{eqn:matrix} is also invertible, 
    the result follows.
\end{proof}

\subsection{Conclusion of the proof}
     Let $[\nabla]\in \Op_n\cap \Mbold_n^\RBbb$. Since $\nabla$ is
    irreducible, $[\nabla]$ is a smooth point of $\Mbold_n$. Now
   since $\Mbold_n^{\RBbb}$ is the fixed point set of an
    involution, the real dimension of $\Mbold_n^{\RBbb}$ at $[\nabla]$ is
    half the real dimension of $\Mbold_n$. As mentioned above, $\Op_n$ is
    also half dimensional. 
      Let 
    $[b]\in T_{[\nabla]}\Op_n\cap  T_{[\nabla]}\Mbold_n^{\RBbb}$.
    Then by Lemma  \ref{lem:real}, $H_{\nabla}(b,b)=0$. But then by
    Proposition \ref{prop:real}, $[b]=0$.
    Hence,   $T_{[\nabla]}\Op_n\cap
    T_{[\nabla]}\Mbold_n^{\RBbb}=\{0\}$, and this proves the transversality  
Theorem \ref{thm:main}.  The discreteness, Corollary \ref{cor:main},
follows immediately.

\subsection*{AI declaration} The authors used AI only  to check for
mistakes and to build a formal verification 
dossier so that the Aristotle system could formally verify their proof. 

\subsection*{Acknowledgements} The authors would like to thank Pavel
Etingof for discussions.

 \bibliography{papers}
 \bibliographystyle{plain}
 \end{document}